\documentclass[11pt]{article}
\usepackage[margin=1.1in,a4paper]{geometry}
\usepackage[dvipsnames]{xcolor}
\usepackage[authoryear,longnamesfirst]{natbib}
\usepackage{amsmath,amssymb,mathtools,amsthm}
\usepackage{amsfonts,bm}
\usepackage{url}
\usepackage[colorlinks=true,citecolor=blue,linkcolor=blue,urlcolor=blue]{hyperref}

\newtheorem{theorem}{Theorem}
\newtheorem{lemma}{Lemma}
\newtheorem{corollary}{Corollary}
\newtheorem{remark}{Remark}

\DeclareMathOperator{\per}{per}
\newcommand{\one}{\mathbf{1}}

\title{Proof of Dittert's conjecture for dimensions $n\ge 17$}
\author{Zhekai Pang\thanks{Universitat Pompeu Fabra, Barcelona, Spain. Email: \texttt{zhekai.pang@upf.edu}.}}
\date{}

\begin{document}

\maketitle

\begin{abstract}
Dittert's conjecture gives a sharp upper bound for the Dittert functional on nonnegative matrices whose entries sum to $n$. It extends the van der Waerden permanent problem from the doubly stochastic polytope to a larger simplex in which row and column sums are allowed to vary. We prove the conjecture for every dimension $n\ge 17$. The proof combines the Knopp--Sinkhorn lower bound for boundary points of the doubly stochastic polytope with a refined scaling step in the Cheon--Wanless method. The main improvement is a subset-sum estimate for the row and column sums of a near maximizer that is sharper for subset sizes in the central range and reduces the scaling factor needed to obtain a doubly superstochastic matrix. This strengthened comparison is sufficient to exclude boundary maximizers in all dimensions $n\ge 17$, and the known positive-support characterization then identifies the unique maximizer as $n^{-1}J_n$.
\end{abstract}



\section{Introduction}

Extremal problems for permanents form a classical part of matrix theory and combinatorics. The van der Waerden conjecture, proved independently by Egorychev and Falikman, identifies the minimum permanent on the doubly stochastic polytope and remains one of the central examples of a sharp permanent inequality \citep{Egorychev1981,Falikman1981}. Dittert's conjecture asks for a related extremal statement on the larger simplex of nonnegative matrices whose entries have fixed sum. The problem was recorded by Minc as Conjecture 28 in his survey on permanents \citep[Conjecture~28]{Minc1983}, and later by Zhan as Conjecture 14 in his list of open problems in matrix theory \citep[Conjecture~14]{Zhan2008}.

The existing results show that the main difficulty is the boundary of the simplex. Sinkhorn proved the case \(n=2\), and Hwang proved the case \(n=3\) \citep{Sinkhorn1984,Hwang1987}. Hwang also proved that the only possible positive maximizing matrix is \(n^{-1}J_n\) \citep{Hwang1986}. Cheon and Wanless later proved a quantitative estimate for near maximizers \citep[Theorem~2.1]{CheonWanless2012}, and two boundary exclusions: partly decomposable matrices cannot maximize the Dittert functional \citep[Theorem~3.2]{CheonWanless2012}, and zeroes forming a single block cannot occur in a maximizing matrix \citep[Theorem~3.7]{CheonWanless2012}. Udayan and Somasundaram prove that every \(\phi\)-maximizing matrix on \(K_4\) is fully indecomposable \citep[Theorem~3.1]{UdayanSomasundaram2024}.

This paper proves Dittert's conjecture for every dimension \(n\ge 17\). The new ingredient is a sharper scaling step. A refined subset-sum estimate for the row and column sums allows us to improve the scaling factor in the Cheon--Wanless method. This reduces the comparison needed against the Knopp--Sinkhorn boundary gap and makes the argument work from dimension seventeen onward.

\section{Preliminaries}

Let \(S_n\) be the symmetric group on \(\{1,\ldots,n\}\). For an \(n\times n\) matrix \(A=(a_{ij})\), define
\[
 \per(A)=\sum_{\sigma\in S_n}\prod_{i=1}^{n} a_{i\sigma(i)} .
\]
Let \(K_n\) be the set of all real nonnegative \(n\times n\) matrices whose entries have sum \(n\). If \(A\in K_n\), let \(r_i=\sum_{j=1}^{n}a_{ij}\) and \(c_j=\sum_{i=1}^{n}a_{ij}\). The Dittert function is
\[
 \phi(A)=\prod_{i=1}^{n} r_i+\prod_{j=1}^{n} c_j-\per(A).
\]
Let \(J_n=\one\one^T\). Dittert's conjecture asserts that \(\phi(A)\le 2-n!/n^n\) for every \(A\in K_n\), with equality only at \(n^{-1}J_n\). This is recorded in Zhan's survey \citep[Conjecture~14]{Zhan2008}.

Let \(\gamma_n=n!/n^n\). Then \(\phi(n^{-1}J_n)=2-\gamma_n\). Let \(\Omega_n\) be the set of all \(n\times n\) doubly stochastic matrices. A nonnegative matrix \(M\) is called doubly superstochastic if there exists a matrix \(B\in\Omega_n\) such that \(B\le M\) entrywise. We use the van der Waerden--Egorychev--Falikman theorem in the form stated by Cheon and Wanless \citep[Theorem~1.1]{CheonWanless2012}: if \(B\in\Omega_n\), then \(\per(B)\ge\gamma_n\), with equality only for \(B=n^{-1}J_n\). Define
\[
 m_n=(n-2)!\left(\frac{n-2}{(n-1)^2}\right)^{n-2}.
\]

\begin{lemma}\label{lem:KS}
Let \(n>3\). If \(B\in\Omega_n\) has at least one zero entry, then \(\per(B)\ge m_n\).
\end{lemma}

\begin{proof}
The matrices in \(\Omega_n\) with at least one zero entry are the boundary points of the doubly stochastic polytope. For \(n>3\), the corollary following \citet[Theorem~5]{KnoppSinkhorn1982} shows that the minimum of the permanent on this boundary is \(m_n\). The assertion follows.
\end{proof}

\begin{lemma}\label{lem:pinsker}
Let \(0<p,q<1\). Then
\[
 p\log \frac{p}{q}+(1-p)\log\frac{1-p}{1-q}\ge 2(p-q)^2.
\]
\end{lemma}

\begin{proof}
The case \(q=p\) is trivial. Assume first that \(q>p\). For fixed \(p\), let
\[
 D(q)=p\log \frac{p}{q}+(1-p)\log\frac{1-p}{1-q}.
\]
Then
\[
 D'(q)=-\frac{p}{q}+\frac{1-p}{1-q}=\frac{q-p}{q(1-q)}.
\]
Since \(D(p)=0\), we have
\[
 D(q)=D(q)-D(p)=\int_{p}^{q}D'(x)\,\mathrm{d}x
      =\int_{p}^{q}\frac{x-p}{x(1-x)}\,\mathrm{d}x.
\]
Since \(x(1-x)\le 1/4\) on \((0,1)\), it follows that
\[
 D(q)\ge \int_{p}^{q}4(x-p)\,\mathrm{d}x=2(q-p)^2.
\]
If \(q<p\), then
\[
 D(q)=D(q)-D(p)=-\int_{q}^{p}D'(x)\,\mathrm{d}x
      =\int_{q}^{p}\frac{p-x}{x(1-x)}\,\mathrm{d}x.
\]
The same estimate gives \(D(q)\ge \int_{q}^{p}4(p-x)\,\mathrm{d}x=2(p-q)^2\). This proves the lemma.
\end{proof}

\begin{lemma}\label{lem:balanced}
Let \(A\in K_n\) satisfy \(\phi(A)\ge\phi(n^{-1}J_n)\), and let \(\delta=\gamma_n-\per(A)\). Assume \(\delta\le1/4\). Then \(0\le\delta\le\gamma_n\). Let \(k\) be an integer with \(0\le k\le n\). For every row index set \(\alpha\subseteq\{1,\ldots,n\}\) with \(|\alpha|=k\),
\begin{equation}\label{eq:balanced-row}
 \left(\sum_{i\in\alpha} r_i-k\right)^2\le \frac{3\delta k(n-k)}{n}.
\end{equation}
The same estimate holds for column sums.
\end{lemma}

\begin{proof}
By \citet[Theorem~2.1]{CheonWanless2012}, we have \(0\le\delta\le\gamma_n\). It is enough to prove the row estimate, because the column estimate follows by transposition. Let \(R=\prod_{i=1}^{n}r_i\) and \(C=\prod_{j=1}^{n}c_j\). Since \(\sum_{i=1}^{n}r_i=n\) and \(\sum_{j=1}^{n}c_j=n\), the AM-GM inequality gives \(R\le1\) and \(C\le1\). By assumption, \(\phi(A)\ge\phi(n^{-1}J_n)\). Thus
\[
 R+C-(\gamma_n-\delta)\ge 2-\gamma_n.
\]
Together with \(C\le1\), this gives
\begin{equation}\label{eq:R-lower}
 R\ge 1-\delta .
\end{equation}

The cases \(k=0\) and \(k=n\) are trivial. Assume \(1\le k\le n-1\), and fix a row index set \(\alpha\) with \(|\alpha|=k\). Let
\[
 \varepsilon=\sum_{i\in\alpha}r_i-k.
\]
We first prove
\begin{equation}\label{eq:aux-eps}
 \varepsilon^2\le \frac{2n\delta}{3}.
\end{equation}
If \(\varepsilon<0\), apply the following argument to the complement of \(\alpha\), whose deviation is \(-\varepsilon\). Hence we may assume \(\varepsilon\ge0\). Since \(R>0\), all row sums are positive and \(0\le \varepsilon<n-k\). By AM-GM applied to the rows in \(\alpha\),
\[
 \prod_{i\in\alpha}r_i\le \left(\frac{1}{k}\sum_{i\in\alpha}r_i\right)^k
 =\left(1+\frac{\varepsilon}{k}\right)^k.
\]
By AM-GM applied to the rows outside \(\alpha\),
\[
 \prod_{i\notin\alpha}r_i\le \left(\frac{1}{n-k}\sum_{i\notin\alpha}r_i\right)^{n-k}
 =\left(1-\frac{\varepsilon}{n-k}\right)^{n-k}.
\]
Multiplying the two inequalities gives
\[
 R\le \left(1+\frac{\varepsilon}{k}\right)^k\left(1-\frac{\varepsilon}{n-k}\right)^{n-k}.
\]
Let \(p=k/n\) and \(q=(k+\varepsilon)/n\). Taking logarithms gives
\[
 \log R\le -n\left(p\log\frac{p}{q}+(1-p)\log\frac{1-p}{1-q}\right).
\]
By Lemma \ref{lem:pinsker}, \(\log R\le -2\varepsilon^2/n\). Combining this with \eqref{eq:R-lower},
\[
 1-\delta\le R\le \exp\left(-\frac{2\varepsilon^2}{n}\right).
\]
Therefore
\[
 \varepsilon^2\le -\frac{n}{2}\log(1-\delta).
\]
Also,
\[
 -\log(1-\delta)=\int_{0}^{\delta}\frac{\mathrm{d}t}{1-t}
 \le \int_{0}^{\delta}\frac{\mathrm{d}t}{1-\delta}
 =\frac{\delta}{1-\delta}.
\]
Because \(\delta\le1/4\), we have \((1-\delta)^{-1}\le4/3\). Hence
\[
 \varepsilon^2\le -\frac{n}{2}\log(1-\delta)
 \le \frac{n\delta}{2(1-\delta)}
 \le\frac{2n\delta}{3}.
\]
This proves \eqref{eq:aux-eps}.

Let \(m=\min\{k,n-k\}\). By \citet[Theorem~2.1]{CheonWanless2012}, applied to \(\alpha\),
\[
 \left|\sum_{i\in\alpha}r_i-k\right|\le \sqrt{2\delta k}.
\]
Since \(\sum_{i=1}^{n}r_i=n\), we have
\[
 \sum_{i\in\alpha^c}r_i-(n-k)
 =n-\sum_{i\in\alpha}r_i-(n-k)
 =k-\sum_{i\in\alpha}r_i
 =-\varepsilon .
\]
Applying \citet[Theorem~2.1]{CheonWanless2012} to the complement \(\alpha^c\) gives
\[
 |\varepsilon|
 =\left|\sum_{i\in\alpha^c}r_i-(n-k)\right|
 \le \sqrt{2\delta(n-k)}.
\]
Together with the bound obtained from \(\alpha\), this gives
\begin{equation}\label{eq:CW-min}
 \varepsilon^2
 =\left|\sum_{i\in\alpha}r_i-k\right|^2
 \le 2\delta\min\{k,n-k\}
 =2\delta m .
\end{equation}
If \(m\le n/3\), then \(2\delta m\le 3\delta m(n-m)/n\). If \(m> n/3\), then \(m(n-m)>2n^2/9\), and \eqref{eq:aux-eps} gives \(\varepsilon^2\le 2n\delta/3\le3\delta m(n-m)/n\). Since \(m(n-m)=k(n-k)\), the desired inequality follows.
\end{proof}

The estimate in Lemma \ref{lem:balanced} is not uniformly stronger than the bound \(2\delta\min\{k,n-k\}\) in \eqref{eq:CW-min}. For \(\delta>0\), it is strictly stronger precisely when \(n/3<k<2n/3\); for \(\delta=0\), both bounds are zero. Its advantage is that the symmetric factor \(k(n-k)\) leads to the improved scaling factor in Lemma \ref{lem:scaling}.

\begin{lemma}\label{lem:scaling}
Let \(A\in K_n\) satisfy \(\phi(A)\ge\phi(n^{-1}J_n)\), and let \(\delta=\gamma_n-\per(A)\). Assume \(\delta<\min\{1/4,1/(3n)\}\). Then \((1-\sqrt{3n\delta})^{-1}A\) is doubly superstochastic.
\end{lemma}

\begin{proof}
For a matrix \(X=(x_{ij})\), let \(s(X)\) denote the sum of its entries, and let \(X[I,J]\) denote the submatrix with row index set \(I\) and column index set \(J\). Thus \(s(X[I,J])=\sum_{i\in I}\sum_{j\in J}x_{ij}\). Let \(t=\sqrt{3n\delta}\), and suppose that \(S=(1-t)^{-1}A\) is not doubly superstochastic. By \citet[Lemma~2.2]{CheonWanless2012}, there exist row and column sets \(I,J\subseteq\{1,\ldots,n\}\) such that, with \(u=|I|\), \(v=|J|\), and \(p=u+v-n\),
\[
 s(S[I,J])<p.
\]
Since \(S=(1-t)^{-1}A\), this is equivalent to
\begin{equation}\label{eq:not-super-new}
 s(A[I,J])<(1-t)p.
\end{equation}
Because \(s(A[I,J])\ge0\) and \(1-t>0\), this inequality forces \(p>0\). As \(p\) is an integer, \(p\ge1\).

Let \(a=n-u=|I^c|\) and \(b=n-v=|J^c|\). Then \(p=n-a-b\). By Lemma \ref{lem:balanced}, applied to \(I^c\),
\[
 \sum_{i\in I^c}r_i-a
 \le \left|\sum_{i\in I^c}r_i-a\right|
 \le \sqrt{\frac{3\delta a(n-a)}{n}}.
\]
Similarly,
\[
 \sum_{j\in J^c}c_j-b
 \le \left|\sum_{j\in J^c}c_j-b\right|
 \le \sqrt{\frac{3\delta b(n-b)}{n}}.
\]
Because all entries are nonnegative,
\begin{align}
 s(A[I,J])&=n-\sum_{i\in I^c}r_i-\sum_{j\in J^c}c_j+s(A[I^c,J^c]) \notag\\
 &\ge p-\sqrt{\frac{3\delta}{n}}\left(\sqrt{a(n-a)}+\sqrt{b(n-b)}\right).\label{eq:submatrix-lower}
\end{align}
For \(0\le a,b\le n\), we have \(\sqrt{a(n-a)}\le n/2\) and \(\sqrt{b(n-b)}\le n/2\). Hence \eqref{eq:submatrix-lower} gives \(s(A[I,J])\ge p-t\). Since \(p\ge1\) and \(0\le t<1\), we have \(p-t\ge(1-t)p\). This contradicts \eqref{eq:not-super-new}. Therefore \(S\) is doubly superstochastic.
\end{proof}

\section{Main result}

\begin{lemma}\label{lem:estimate}
For every integer \(n\ge17\),
\begin{equation}\label{eq:gap-new}
 m_n-\gamma_n>\frac34 n^3m_n^2.
\end{equation}
\end{lemma}

\begin{proof}
Let \(q_n=m_n/\gamma_n\). From the definitions of \(m_n\) and \(\gamma_n\),
\begin{equation}\label{eq:qn-new}
 q_n=\frac{n^{n-1}(n-2)^{n-2}}{(n-1)^{2n-3}}=\left(1-\frac1{(n-1)^2}\right)^{n-1}\frac{n-1}{n-2}.
\end{equation}
The desired inequality is equivalent to
\begin{equation}\label{eq:q-equiv-new}
 \frac{q_n-1}{q_n^2}>\frac34 n^3\gamma_n .
\end{equation}

We first check \(n=17\). Direct rational arithmetic gives
\[
 \frac{q_{17}-1}{q_{17}^2}
 >\frac{19}{10000}
 >\frac34\,17^3\gamma_{17}.
\]
Hence \eqref{eq:q-equiv-new} holds for \(n=17\).

Now we prove \eqref{eq:q-equiv-new} for \(n\ge18\). Let \(x=(n-1)^{-1}\). Then \(0<x\le1/17\). Taking logarithms in \eqref{eq:qn-new},
\begin{align}
 \log q_n
 &=\frac1x\log(1-x^2)-\log(1-x) \notag\\
 &=-\sum_{r=1}^{\infty}\frac{x^{2r-1}}{r}+\sum_{r=1}^{\infty}\frac{x^r}{r} \notag\\
 &=\sum_{j=1}^{\infty}\frac{x^{2j}}{2j}+\sum_{j=1}^{\infty}x^{2j+1}\left(\frac1{2j+1}-\frac1{j+1}\right) \notag\\
 &=\sum_{j=1}^{\infty}x^{2j}\left(\frac1{2j}-\frac{jx}{(2j+1)(j+1)}\right).\label{eq:logseries-new}
\end{align}
Since for all \(j\ge1\),
\[
 \frac{jx}{(2j+1)(j+1)}
 <\frac{j}{(2j+1)(j+1)}
 <\frac1{2j},
\]
every summand in \eqref{eq:logseries-new} is positive. Keeping the first summand gives
\[
 \log q_n>x^2\left(\frac12-\frac{x}{6}\right)>\frac{x^2}{3}.
\]
Since \(e^y-1>y\) for \(y>0\), applied with \(y=\log q_n\),
\begin{equation}\label{eq:qminus-new}
 q_n-1>\log q_n>\frac{x^2}{3}=\frac1{3(n-1)^2}.
\end{equation}
Also \eqref{eq:qn-new} gives \(q_n<(n-1)/(n-2)\le17/16\). Hence
\begin{equation}\label{eq:qq-new}
 \frac{q_n-1}{q_n^2}>\frac{1}{3(n-1)^2}\left(\frac{16}{17}\right)^2.
\end{equation}
It is enough to prove
\begin{equation}\label{eq:gamma-new}
 \gamma_n<\frac{4}{9n^3(n-1)^2}\left(\frac{16}{17}\right)^2.
\end{equation}
Let
\[
 T_n=\frac94 n^3(n-1)^2\gamma_n\left(\frac{17}{16}\right)^2.
\]
Then \eqref{eq:gamma-new} is equivalent to \(T_n<1\). Since \(\gamma_{n+1}/\gamma_n=(n/(n+1))^n\),
\[
 \frac{T_n}{T_{n+1}}=\left(\frac{n-1}{n}\right)^2\left(1+\frac1n\right)^{n-3}.
\]
By Bernoulli's inequality, \((1+1/n)^{n-3}>1+(n-3)/n=(2n-3)/n\). For \(n\ge18\), \((2n-3)(n-1)^2>n^3\). Therefore \(T_n/T_{n+1}>1\), so \(T_n\) is decreasing for \(n\ge18\). A direct calculation gives
\[
 T_{18}=\frac{1243617249875}{1785233613312}<1.
\] Thus \(T_n<1\) for every \(n\ge18\). This proves \eqref{eq:gamma-new}. Combining \eqref{eq:gamma-new} and \eqref{eq:qq-new} gives \eqref{eq:q-equiv-new}. Hence \eqref{eq:gap-new} holds for every \(n\ge17\).
\end{proof}

\begin{remark}\label{rem:threshold}
The threshold in Lemma \ref{lem:estimate} is sharp for this comparison. If \(\Lambda_n=(m_n-\gamma_n)/(\frac34 n^3m_n^2)\), then \(\Lambda_{16}\approx0.622934<1\), while \(\Lambda_{17}\approx1.205958>1\).
\end{remark}

\begin{theorem}\label{thm:main}
Let \(n\ge17\). Then \(\phi(A)\le\phi(n^{-1}J_n)\) for every \(A\in K_n\). Equality holds if and only if \(A=n^{-1}J_n\).
\end{theorem}

\begin{proof}
Because \(K_n\) is compact and \(\phi\) is continuous, there exists \(A\in K_n\) such that \(\phi(A)=\max_{X\in K_n}\phi(X)\).
If \(A\) is positive, then Hwang's positive-support result \citep{Hwang1986} gives \(A=n^{-1}J_n\).
We now assume that \(A\) has a zero entry. Since \(A\) maximizes \(\phi\), we have \(\phi(A)\ge\phi(n^{-1}J_n)=2-\gamma_n\). Let \(\delta=\gamma_n-\per(A)\). By \citet[Theorem~2.1]{CheonWanless2012}, \(0\le\delta\le\gamma_n\).

We first show that \(\delta>0\). If \(\delta=0\), then \(\per(A)=\gamma_n\). By the AM-GM inequality, \(\prod_{i=1}^{n}r_i\le1\) and \(\prod_{j=1}^{n}c_j\le1\). Since \(\phi(A)\ge2-\gamma_n\), equality must hold in both inequalities. Hence all row sums and all column sums of \(A\) are equal to \(1\), so \(A\in\Omega_n\). By \citet[Theorem~1.1]{CheonWanless2012}, \(A=n^{-1}J_n\), contradicting the fact that \(A\) has a zero entry. Thus \(\delta>0\).

For \(n\ge17\), we have \(\delta\le\gamma_n<\min\{1/4,1/(3n)\}\). This follows from \(\gamma_{17}<1/2320000\) and from the facts that \(\gamma_n\) and \(n\gamma_n\) are decreasing in \(n\). Therefore Lemma \ref{lem:scaling} applies. Let \(t=\sqrt{3n\delta}\) and \(S=(1-t)^{-1}A\). Then \(S\) is doubly superstochastic, so there exists \(B\in\Omega_n\) such that \(0\le B\le S\) entrywise.

The matrix \(S\) is a positive scalar multiple of \(A\), so \(S\) has a zero entry. At the same position, \(0\le B\le S\) forces \(B\) to have a zero entry. By Lemma \ref{lem:KS}, \(\per(B)\ge m_n\). Because the permanent is monotone on nonnegative matrices, \(\per(S)\ge\per(B)\ge m_n\). Since \(S=(1-t)^{-1}A\), we obtain
\begin{equation}\label{eq:firstineq-new}
 \gamma_n-\delta=\per(A)=(1-t)^n\per(S)\ge (1-t)^n m_n .
\end{equation}

Now we prove that \eqref{eq:firstineq-new} is impossible. Bernoulli's inequality gives \((1-t)^n\ge1-nt\). Hence
\[
\begin{aligned}
 (1-t)^n m_n-(\gamma_n-\delta)
 &\ge m_n-\gamma_n+\delta-nm_n\sqrt{3n\delta} \\
 &=m_n-\gamma_n+\left(\sqrt{\delta}-\frac12 nm_n\sqrt{3n}\right)^2-\frac34 n^3m_n^2 \\
 &\ge m_n-\gamma_n-\frac34 n^3m_n^2.
\end{aligned}
\]
By Lemma \ref{lem:estimate}, the last expression is positive. Hence \((1-t)^n m_n>\gamma_n-\delta\), contradicting \eqref{eq:firstineq-new}. Therefore a maximizing matrix cannot have a zero entry.

Every maximizing matrix is positive. Hence Hwang's positive-support result \citep{Hwang1986} implies that every maximizing matrix is \(n^{-1}J_n\). This proves both the inequality and the equality statement.
\end{proof}

\begin{corollary}\label{cor:known}
Dittert's conjecture is true for \(n=2,3\), and for every \(n\ge17\).
\end{corollary}

\begin{proof}
The case \(n=2\) is due to Sinkhorn \citep{Sinkhorn1984}, and the case \(n=3\) is due to Hwang \citep{Hwang1987}. The cases \(n\ge17\) follow from Theorem \ref{thm:main}.
\end{proof}

\begin{remark}
The present argument does not settle the cases \(4\le n\le16\).
\end{remark}

\end{document}